\documentclass[12pt]{amsart}
\usepackage{latexsym, amsmath, amscd, amssymb, amsthm} 
\newtheorem{theorem}{Theorem}[section]

\begin{document}

\noindent

 \title[]{ the Poincar\'e series for the algebra of invariants of $n$-ary form }

\author{Leonid Bedratyuk}\address{Khmelnitskiy national university, Insituts'ka, 11,  Khmelnitskiy, 29016, Ukraine}

\begin{abstract} 
The formula for the Poincare series of the algebra of invariant of $n$-ary form is found.
\end{abstract}
\maketitle

\section{Introduction}
Let $V_{n,d}$ be the  vector $\mathbb{C}$-space of  $n$-ary forms of degree $d$ endowed with the natural action of the algebra  $\mathfrak{sl_{n}}.$  Denote  by  $\mathcal{I}_{n,d}:=\mathbb{C}[V_{n,d}]^{\,\mathfrak{sl_{n}}}$ the  algebra of   $\mathfrak{sl_{n}}$-invariant polynomial functions. The algebra $\mathcal{I}_{n,d}$  is  a graded algebra
$$
\mathcal{I}_{n,d}=(\mathcal{I}_{n,d})_0 \oplus (\mathcal{I}_{n,d})_1 \oplus \cdots \oplus (\mathcal{I}_{n,d})_k \oplus  \cdots,
$$
here  $(\mathcal{I}_{n,d})_k$ is the subspace of   homogeneous  invariants of degree $k.$
Denote  ${\nu_{n,d}(k) :=\dim (\mathcal{I}_{n,d})_k.}$  

In \cite{B_SC} the formula for $\nu_{n,d}(k)$ was found. 
In the preprint we derive  a  formula for the Poincar\'e series
$$
\mathcal{P}_{n,d}(t)=\sum_{i=0}^{\infty}\nu_{n,d}(k) t^k,
$$
 of the algebra invariants $\mathcal{I}_{n,d}.$ 
\section{Poincar\'e series}

Consider    the  Lie algebra    $\mathfrak{sl_{n}}$  and  let $E_{k,i}$ denote the matrix that has a one in the $k$-th row and $i$-th column and that has zeros elsewhere. Let 
$$
\mathfrak{h}=\{e_1 E_{1,1} +e_2 E_{2,2}+\cdots+e_n E_{n,n} \mid e_1+e_2+\cdots+e_n=0, e_i \in \mathbb{C}\},
$$ 
be  the Cartan subalgebra of $\mathfrak{sl_{n}}$.
Define $L_i \in \mathfrak{h}^*$ by  $L_i(E_{j,j})=\delta_{i,j}.$ Let $\alpha_{i,j}=L_i-L_{j},$ ${ 1\leqslant i<j \leqslant n}$ are the positive roots  $\mathfrak{sl_{n}}$  and let  $\phi_i=L_1+L_2+\ldots+L_{i}, i=1,\ldots, n-1$ are the fundamental weights. 
 The matrices  $$H_1:=E_{1,\,1}{-}E_{2,\, 2}, H_2:=E_{2,\,2}{-}E_{3,\,3}, \ldots H_{n-1}:=E_{n-1,\,n-1}{-}E_{n,\,n}$$  generate the Cartan  subalgebra of the  Lie  algebra   $\mathfrak{sl_{n}}.$ 
It is easy to check that  $\phi_i(H_j)=\delta_{i,j}.$
 Denote by  $\lambda=(\lambda_1,\lambda_2,\ldots,\lambda_{n-1})$  the weight 
$$
\lambda_1 \phi_1+\lambda_2 \phi_2+\ldots+\lambda_{n-1} \phi_{n-1}, \lambda_i \in \mathbb{Z}.
$$
In the notation the half of sum of all positive roots  $\rho $ equals   $(1,1,\ldots,1).$

In \cite{B_SC} we proved that 
\begin{equation}
\nu_{n,d}(k)=\sum_{s \in \mathcal{W}} (-1)^{|s|} c_{n,d}\bigl(k,\{\rho-s(\rho)\}\bigr),
\end{equation}
where   $\mathcal{W}$  is the Weyl group of $\mathfrak{sl}_n,$ $\{\mu \}$ is the unique dominant weight on the orbit $\mathcal{W}(\mu), $ $c_{n,d}(k,\mu) :=c_{n,d}(k,(\mu_1,\mu_2,\ldots,\mu_{n-1}))$ is the number    of non-negative integer solutions of the system of equations  for the indeterminates $ \alpha_{\textbf{\em i}}$ 
\begin{gather}
\left \{
\begin{array}{l}
\omega_2(\alpha)-\omega_1(\alpha)=\mu_1, \\
\ldots \\
\omega_{n-1}(\alpha)-\omega_{n-2}(\alpha)=\mu_{n-2},\\
 \omega_1(\alpha)+\omega_2(\alpha)+\cdots+\omega_{n-1}(\alpha) +2\,\omega_{n-1}(\alpha)=k\, d-\mu_{n-1}, \\
|\alpha|=k.
\end{array}
\right.
\end{gather}
$\omega_s(\alpha):=\sum_{\textbf{\em i} \in \textbf{I}_{n,d}} i_s\, \alpha_{\textbf{\em i}},$ $\textbf{I}_{n,d}:=\{ \textbf{\em i}:=(i_1,i_2,\ldots, i_{n-1})  \in \mathbb{Z}_{+}^{n-1},$ $   |\textbf{\em i}|\leq d \}, $ $  |\textbf{\em i}|:=i_1+\cdots+ i_{n-1}.$

  Let us derive now  the formula  for calculation of  $\nu_{n,d}(k).$  Solving the system  $(2)$
for  $\omega_1(\alpha),$ $ \omega_2(\alpha), $ $\ldots, $ $\omega_{n-1}(\alpha) $  we  get  
\begin{equation}
\left \{
\begin{array}{l}
\displaystyle \omega_1(\alpha)= \frac{k\,d}{n}-\left(\sum_{s=1}^{n-2} \mu_s-\frac{1}{n}\Bigl(\,\sum_{s=1}^{n-2} s \mu_s-\mu_{n-1}\Bigr) \right)\\
\displaystyle \omega_2(\alpha)= \frac{k\,d}{n} -\left(\sum_{s=2}^{n-2} \mu_s-\frac{1}{n}\Bigl(\,\sum_{s=1}^{n-2} s \mu_s-\mu_{n-1}\Bigr)\right)\\
\ldots \\
\displaystyle \omega_{n-1}(\alpha)= \frac{k\,d}{n}-\frac{1}{n}\Bigl(\,\mu_{n-1}-\sum_{s=1}^{n-2} s \mu_s \Bigr),\\
|\alpha|=k,
\end{array}
\right.
\end{equation}
It is not hard to prove that  the number  $c_{n,d}(k,(0,0,\ldots,0))$ of   non-negative integer solutions of the  system 
 $$
\left \{
\begin{array}{l}
\displaystyle \omega_1(\alpha)= \frac{k\,d}{n}\\
\displaystyle \omega_2(\alpha)= \frac{k\,d}{n} \\
\ldots \\
\displaystyle \omega_{n-1}(\alpha)= \frac{k\,d}{n}\\
|\alpha|=k
\end{array}
\right.
$$
is equal to the coefficient of  $\displaystyle t^k (q_1 q_2 \ldots q_{n-1})^{\frac{k\,d}{n}}$    of the  expansion   of the  series
$$
R_{n,d}(t,q_1, \ldots q_{n-1})=\frac{1}{\displaystyle \prod_{0 \leqslant |\eta|  \leqslant  d } \left(1-t q_1^{\eta_1} q_2^{\eta_2} \cdots q_{n-1}^{\eta_{n-1}}\right)}, \eta=(\eta_1,\eta_2, \ldots, \eta_{n-1}) \in \mathbb{Z}_{+}^{n-1}.
$$
Denote it in such a way:
$$c_{n,d}(k,(0,0,\ldots,0))=\left[t^k (q_1q_2 \cdots q_{n-1} )^{\frac{k\,d}{n}}\right] R_{n,d}(t,q_1 q_2 \ldots q_{n-1}).$$
Then, for a set of nonnegative integer numbers  $\mu:=(\mu_1,\mu_2,\ldots,\mu_{n-1})$ the number  $$c_{n,d}(k,(\mu_1,\mu_2,\ldots,\mu_{n-1})),$$ of integer nonnegative solutions of the system $(4)$
   equals
\begin{gather*}
c_{n,d}(k,(\mu_1,\ldots,\mu_{n-1}))=\left[t^k q_1^{\frac{k\,d}{n}-\mu_1'}\cdots q_{n-1}^{\frac{k\,d}{n}-\mu_{n-1}'} \right]R_
{n,d}(t,q_1 \ldots q_{n-1})=\\
= \left[t^k (q_1\cdots q_{n-1})^{\frac{k\,d}{n}}\right] q_1^{\mu'_1} \cdots q_{n-1}^{\mu_{n-1}'} R_{n,d}(t,q_1 \ldots q_{n-1}).
\end{gather*}
Here 
\begin{equation}
\displaystyle \mu_i'=\left(\sum_{s=i}^{n-2} \mu_s-\frac{1}{n}\Bigl(\,\sum_{s=1}^{n-2} s \mu_s-\mu_{n-1}\Bigr) \right), i=1,\ldots, n-1.\\
\end{equation}
By using the multi-index notation $\textit{\textbf{q}}^{\mu}:=q_1 ^{\mu_1} \cdots q_{n-1} ^{\mu_{n-1}}$, rewrite the expression for  $c_{n,d}(k,\mu)$ in the  form: 
$$
c_{n,d}(k,\mu)=\left[t^k \textit{\textbf{q}}^{\frac{k\,d}{n}}\right] \left( \textit{\textbf{q}}^{\mu'}  R_{n,d}(t,\textit{\textbf{q}}) \right).
$$

Then, Theorem 2.5 implies the  following formula:
\begin{gather}
\nu_{n,d}(k)=\left[t^k \textit{\textbf{q}}^{\frac{k\,d}{n}}\right]  \frac{\displaystyle \sum_{s \in W} (-1)^{|s|} \textit{\textbf{q}}^{\{\rho-s(\rho)\}'}}{ \displaystyle \prod_{|\eta| \leq  d } \left(1-t  \textit{\textbf{q}}^{\eta}\right)}.
\end{gather}
Let us recall that the Poincar\'e series  $\mathcal{P}_{n,d}(t)$ of the algebra  of invariants $ \mathcal{I}_{n,d}$ is the ordinary  generating function of the sequence $\nu_{n,d}(k),$ $k=1,2,\ldots .$ To simplify the notation put 
$$
\oint_{|q_{n-1}|=1} \ldots  \oint_{|q_{1}|=1} f(t,q_1,q_2,\ldots,q_{n-1})  \frac{dq_1 \ldots dq_{n-1}}{q_1\ldots q_{n-1}}:=  \oint_{| \textit{\textbf{q}}|=1} f(t,\textit{\textbf{q}})  \frac{d \textit{\textbf{q}}}{ \textit{\textbf{q}}}.
$$

\begin{theorem} The Poincare series $\mathcal{P}_{n,d}(t)$ of the algebra $\mathcal{I}_{n,d}$ equals
\begin{gather}
P_{n,d}(t)=  \oint_{| \textit{\textbf{q}}|=1}   \frac{\displaystyle  \sum_{s \in W} (-1)^{|s|} \textit{\textbf{q}}^{n \{\rho-s(\rho)\}'}  }{\displaystyle \prod_{|\eta| \leq  d } \left(1-t \textit{\textbf{q}}^{n \eta-d \rho}\right)} \frac{d \textit{\textbf{q}}}{ \textit{\textbf{q}}}. 
\end{gather}
\end{theorem}
\begin{proof}
We have 
\begin{gather*}
\nu_{n,d}(k)=\left[t^k \textit{\textbf{q}}^{\frac{k\,d}{n}}\right]  \frac{\displaystyle \sum_{s \in W} (-1)^{|s|}  \textit{\textbf{q}}^{\{\rho-s(\rho)\}'}}{ \displaystyle \prod_{|\eta| \leq  d } \left(1-t  \textit{\textbf{q}}^{\eta}\right)}
=\left[(t \textit{\textbf{q}} ^d)^{k}\right]  \frac{\displaystyle \sum_{s \in W} (-1)^{|s|}  \textit{\textbf{q}}^{n \{\rho-s(\rho)\}'}}{ \displaystyle \prod_{|\eta| \leq  d } \left(1-t  \textit{\textbf{q}}^{n \eta}\right)}=\\
=\left[t^{k}\right]  \frac{\displaystyle \sum_{s \in W} (-1)^{|s|} \textit{\textbf{q}}^{n \{\rho-s(\rho)\}'}}{\displaystyle \prod_{|\eta| \leq  d } \left(1-t q_1^{n \eta_1-d} q_2^{n\eta_2-d} \cdots q_{n-1}^{n\eta_{n-1}-d}\right)}=\\
=\left[t^{k}\right]\oint_{|q_{n-1}|=1} \ldots  \oint_{|q_{1}|=1}   \frac{\displaystyle \sum_{s \in W} (-1)^{|s|}  \textit{\textbf{q}}^{n \{\rho-s(\rho)\}'}}{\displaystyle \prod_{|\eta| \leq  d } \left(1-t  \textit{\textbf{q}}^{n \eta - d \rho}\right)} \frac{dq_1 \ldots dq_{n-1}}{q_1\ldots q_{n-1}}.
\end{gather*}
Therefore
\begin{gather*}
P_{n,d}(t)=\sum_{k=0}^{\infty}\nu_{n,d}(k)t^k=\\
=\sum_{k=0}^{\infty}\left(\left[t^{k}\right] \displaystyle  \oint_{| \textit{\textbf{q}}|=1}   \frac{\displaystyle  \sum_{s \in W} (-1)^{|s|} \textit{\textbf{q}}^{\{\rho-s(\rho)\}'}  }{\displaystyle \prod_{|\eta| \leq  d } \left(1-t q_1^{n \eta_1-d} q_2^{n\eta_2-d} \cdots q_{n-1}^{n\eta_{n-1}-d}\right)}  \frac{d \textit{\textbf{q}}}{ \textit{\textbf{q}}} \right)t^k=\\
= \oint_{| \textit{\textbf{q}}|=1}   \frac{\displaystyle  \sum_{s \in W} (-1)^{|s|} \textit{\textbf{q}}^{n \{\rho-s(\rho)\}'}  }{\displaystyle \prod_{|\eta| \leq  d } \left(1-t \textit{\textbf{q}}^{n\eta-d \rho}\right)} \frac{d \textit{\textbf{q}}}{ \textit{\textbf{q}}}.
\end{gather*}
\end{proof}

\section{Examples}

Let us consider  the case  of binary form. We have $\mathfrak{h}=\langle H_1\rangle,$ where $H_1= E_{1,1}-E_{2,2}.$  There exist the  positive root $\alpha=L_1-L_2=2 L_1$ and the  fundamental weight $\phi_1=L_1.$ The half the  positive root $\rho$ is equal to  $L_1.$ The Weyl group is generated by the reflecsion  $s_{\alpha},$ $(-1)^{s_{\alpha}}=-1.$  The orbit of weight  $\rho$ consists of the two weights $\phi_1$ and $-\phi_1$  and we have 
$
\rho-\mathcal{W}(\rho)=\{ 0, 2 \phi_1\}.
$
Therefore
$$
\nu_{2,d}(k)= c_{2,d}(k,0)- c_{2,d}(k,2), 
$$
where $c_{2,d}(k,m)$is the number of nonnegative integer solutions of the equation $$\alpha_1+2\alpha_2+\cdots + d\, \alpha_d=\displaystyle \frac{d\,k-m}{2}=\displaystyle \frac{d\,k}{2}-1,$$  on the assumption that  $ |\alpha|=k.$ It is exactly the Sylvester-Cayley formula. By $(4),$ $(5)$ we have 
$$
\nu_{2,d}(k)=[t z^{\frac{d\,k}{2}}] \frac{1-z}{(1-t)(1-tz)\ldots (1-tz^d)}=[t^k] \frac{1-z^2}{(1-t z^{-d})(1-t z^{-d+2}) \ldots (1-t z^d)}.
$$
 Theorem 2.1 implies  
\begin{gather}
P_{2,d}(t)=\oint_{|z|=1} \frac{1-z^2}{(1-t z^{-d})(1-t z^{-d+2}) \ldots (1-t z^d)} \frac{dz}{z}.
\end{gather}
It is a well known formula. For instance, in \cite{DerK} it derived from the Molien-Weyl  formula.

Let us now  consider  the case  of ternary form.  We have $\phi_1=L_1, \phi_2=L_1+L_2.$ The positive roots are 
$$
\begin{array}{l}
\alpha_1:=L_1-L_2=2\phi_1-\phi_2=(2,-1),\\
\alpha_2:=L_2-L_3=-\phi_1+2\phi_2=(-1,2),\\
\alpha_3:=L_1-L_3=\phi_1+\phi_2=(1,1).
\end{array}
$$
Then  half the sum of the positive roots $\rho$ is equal to  $(1,1).$ The Weyl group of Lie  algebra  $\mathfrak{sl_{3}}$ is  generated by the  three reflections  $s_{\alpha_1},$ $s_{\alpha_2},$ $s_{\alpha_3}.$  The orbit  $\mathcal{W}(\rho)$ consists of  $6$ weights  -- (1,1) and  
$$
\begin{array}{ll}
 s_{\alpha_1}(1,1)=(-1,2), & (-1)^{|s_{\alpha_1}|}=-1,\\
 s_{\alpha_2}(1,1)=(2,-1), &  (-1)^{| s_{\alpha_2}|}=-1,\\
 s_{\alpha_3}(1,1)=(-1,-1), &  (-1)^{| s_{\alpha_3}|}=-1,\\
s_{\alpha_1} s_{\alpha_3}(1,1)=(1,-2), &  (-1)^{|s_{\alpha_1} s_{\alpha_3}|}=1,\\
s_{\alpha_3} s_{\alpha_1}(1,1)=(-2,1), &  (-1)^{|s_{\alpha_3} s_{\alpha_1}|}=1,\\
\end{array}
$$
Therefore
$$
\rho-\mathcal{W}(\rho)=\{ (0,0), (2,-1),(-1,2), (2,2), (0,3), (3,0)\}.
$$
By  $(1)$  we  obtain
$$
\nu_{3,d}(k)= c_{3,d}(k,(0,0))- 2\,c_{3,d}(k,(1,1)) - c_{3,d}(k, (2,2))+c_{3,d}(k,(0,3))+ c_{3,d}(k, (3,0)).
$$
We have 
$$
R_{3,d}(t,p,q)=\frac{1}{\displaystyle \prod_{0 \leqslant  \mu_1+\mu_2 \leqslant d} (1-t p^{\mu_1} q^{\mu_2})}=\frac{1}{\displaystyle \prod_{k=0}^{d} \prod_{i=0}^{k}(1-t p^{i} q^{k-i})}.
$$
By $(4)$ we get  
$$
(\mu_1,\mu_2)'=\left(\frac{2\mu_1+\mu_2}{3}, \frac{\mu_2-\mu_1}{3} \right).
$$
It implies that 
 $(0,0)'=(0,0),$ $  (1,1)'=(1,0),$ $ (2,2)'=(2,0),$ $ (0,3)'=(1,1),$ $ { (3,0)'=(2,-1)}.$ Therefore
$$
\begin{array}{ll}
\displaystyle c_{3,d}(k,(0,0))=[t^k (pq)^{\frac{d k}{3}}]R_{3,d}(t,p,q), &
c_{3,d}(k,(1,1))=[t^k (pq)^{\frac{d k}{3}}]pR_{3,d}(t,p,q),\\
& \\
\displaystyle c_{3,d}(k,(2,2))=[t^k (pq)^{\frac{d k}{3}}]p^2R_{3,d}(t,p,q), &
c_{3,d}(k,(0,3))=[t^k (pq)^{\frac{d k}{3}}] p q R_{3,d}(t,p,q),\\
& \\
\displaystyle c_{3,d}(k,(3,0))=[t^k (pq)^{\frac{d k}{3}}]\frac{p^2}{q}R_{3,d}(t,p,q). &
\end{array}
$$ 
Thus
$$
\nu_{3,d}(k)=\left[{t^k(pq)^{\frac{d\,k}{3}}}\right] \frac{\displaystyle 1+p\,q+\frac{p^2}{q}-2\,p-p^2}{\displaystyle \prod_{k=0}^{d} \prod_{i=0}^{k}(1-t p^{i} q^{k-i})}=[t^k] \frac{\displaystyle 1+q^3\,p^3+\frac{p^6}{q^3}-2\,p^3-p^6}{\displaystyle \prod_{k=0}^{d} \prod_{i=0}^{k}(1-t p^{3i-d} q^{3(k-i)-d})}. 
$$
By $(6)$ we have 
$$
\mathcal{P}_{3,d}(t)=\oint_{|p|=1}\oint_{|q|=1} \frac{\displaystyle 1+q^3\,p^3+\frac{p^6}{q^3}-2\,p^3-p^6}{\displaystyle \prod_{k=0}^{d} \prod_{i=0}^{k}(1-t p^{3i-d} q^{3(k-i)-d})} \frac{dq}{q} \frac{dp}{p}.
$$

\end{document}